\documentclass[5p,times]{elsarticle}

\usepackage{amsmath,amssymb,amsthm}
\usepackage{booktabs}
\usepackage{graphicx}
\usepackage[hypertexnames=false,colorlinks=true,breaklinks=true,bookmarks=false,urlcolor=blue,citecolor=blue,linkcolor=blue,bookmarksopen=false,draft=false]{hyperref}

\makeatletter\renewcommand*{\top}{%
{\mathpalette\@transpose{}}%
}
\newcommand*{\@transpose}[2]{%
\scriptsize
\raisebox{\depth}{$\m@th#1\mathsf{T}$}%
}
\makeatother

\newtheorem{theorem}{Theorem}
\newtheorem{lemma}[theorem]{Lemma}
\newtheorem{corollary}[theorem]{Corollary}
\newtheorem{proposition}[theorem]{Proposition}
\theoremstyle{remark}
\newtheorem{remark}[theorem]{Remark}

\newcommand{\mleft}{\mathopen{}\mathclose\bgroup\left}
\newcommand{\mright}{\aftergroup\egroup\right}
\newcommand{\R}{\mathbb{R}}
\newcommand{\Vol}{\operatorname{Vol}}
\newcommand{\one}{\mathbf{1}}
\newcommand{\E}{\mathbb{E}}
\newcommand{\Unif}{\operatorname{Unif}}
\newcommand{\VP}{V_{\mathcal{P}}}
\newcommand{\VQ}{V_{\mathcal{Q}}}

\journal{Operations Research Letters}

\begin{document}

\begin{frontmatter}

\title{Volumes of the strong and weak relaxations of the uncapacitated facility-location polytope}

\author{Jon Lee\fnref{fund}}
\ead{jonxlee@umich.edu}
\affiliation{University of Michigan}
\fntext[fund]{J.~Lee was supported in part by ONR grant N00014-24-1-2694.}

\begin{abstract}
Lee, Skipper, and Speakman (2018) asked for formulae or asymptotic expressions for the
volumes of the strong and weak relaxations $\mathcal{P}_{m,n}\cap \mathcal{D}_{m,n}$ and
$\mathcal{Q}_{m,n}\cap \mathcal{D}_{m,n}$ of the uncapacitated facility-location polytope with $m$
facilities and $n$ customers (their Problem~3, ``Texas Hot''). We give exact
formulae for all $m$ and $n$, and we give asymptotics that hold uniformly as
$m+n\to\infty$, whatever the relative growth of $m$ and $n$. The volume ratio
satisfies
$\Vol(\mathcal{Q}_{m,n}\cap \mathcal{D}_{m,n})/\Vol(\mathcal{P}_{m,n}\cap \mathcal{D}_{m,n})\sim
\prod_{j=2}^{n}\mleft(1+\tfrac{1}{(m-1)j}\mright)^{m}$. From this, we read off
the behavior in every growth regime of $m$, and we identify a sharp transition
at $m\asymp\ln n$. Below it, in a uniformly random point of the strong
relaxation, each facility variable is close to $1$ with high probability, and
the ratio grows superlinearly in $n$. Above it, in a uniformly random point of
either relaxation, each facility variable is asymptotically uniform on
$[0,1]$, and the ratio grows linearly in $n$.
\end{abstract}


\end{frontmatter}

\section{Introduction}
Let $M:=\{1,\dots,m\}$ and $N:=\{1,\dots,n\}$ with $m\ge 2$. Following~\cite{LSS}, we let $\mathcal{P}_{m,n}$ be the set of $(x,y)\in\R^{m\times n}\times\R^m$
satisfying
\[
 0\leq x_{ij}\le y_i\,,\quad \forall (i,j)\in M\times N;
  \qquad y_i\le 1\quad\forall i\in M,
\]
and we let $\mathcal{Q}_{m,n}$ be the set of $(x,y)$ satisfying
\begin{gather*}
  \sum_{j\in N}x_{ij}\le n y_i\,,\quad\forall i\in M;\\
  0\le x_{ij}\le 1,\quad\forall (i,j)\in M\times N;\qquad y_i\le 1,\quad\forall i\in M.
\end{gather*}
These are the strong and weak formulations of fixed-charge constraints; Lee
and Morris~\cite{LM} compute their full-dimensional volumes. Imposing the
demand constraints
\[
  \mathcal{D}_{m,n}:=\mleft\{(x,y):\ \sum_{i\in M}x_{ij}=1\ \ \forall j\in N\mright\}
\]
yields relaxations of the classical \emph{uncapacitated
facility-lo\-cation polytope}, which is the convex hull of the solutions (to
either relaxation) having $y\in\{0,1\}^m$.

\medskip
\noindent\textbf{Problem~3 (``Texas Hot'') of~\cite{LSS}; also see~\cite[Section~4]{LM}.}
\emph{Give formulae or asymptotic expressions for the volumes of
$\mathcal{Q}_{m,n}\cap \mathcal{D}_{m,n}$ and $\mathcal{P}_{m,n}\cap \mathcal{D}_{m,n}\,$.}
\medskip

Lee and Morris~\cite{LM} treat the case $m=2$. Here, we handle general $m$ and
$n$.

\subsection*{Motivation}

It is well-known folklore, and easily verified, that  computationally, the weak formulation is much weaker than the strong formulation, which in turn is not much weaker than the convex hull.
 \cite{L07} demonstrated that the computational behavior of these two relaxations can critically depend on the type of objective function that one might be seeking to optimize on these polytopes.
 \cite{LM} proposed and made some headway on comparing these polytopes volumetrically, motivated by the setting on mixed-integer nonlinear optimization.
 In particular, they analyzed the general situation without the demand constraints, and also with the demand constraints but only for $m=2$, establishing the asymptotic behavior in various regimes on $m$ and $n$. In what follows, we make a full volumetric analysis, including the demand constraints, for all $m$ and $n$.

\subsection*{Notation}
We write $\R$ for the real numbers, $\mathbb{Z}$ for the integers, and
$\R^m_{\ge0}\,$ for the nonnegative orthant. For a positive integer $t$, we let
$[t]:=\{1,\dots,t\}$. We use $\subset$ for any
subset, not necessarily proper, and we write $|S|$ for the cardinality of a
finite set $S$. For $a\in\R$, we let $a_+:=\max\{a,0\}$. We write $\ln$ for
the natural logarithm and $\one\{A\}$ for the indicator of a condition $A$.
For $x\in\R^{m\times n}$ and $j\in N$, we write $x_{\cdot j}:=(x_{ij})_{i\in M}$
for the $j$th column of $x$. For $z\in\R^m$ and $\alpha\in\mathbb{Z}^m_{\ge0}\,$,
we let $z^\alpha:=\prod_i z_i^{\alpha_i}$ and $|\alpha|:=\sum_i\alpha_i\,$. We
write $\Gamma$ for Euler's gamma function, $B(a,b):=\Gamma(a)\Gamma(b)/\Gamma(a+b)$
for the beta function, $\gamma\approx0.5772$ for Euler's constant, and
$H_n:=\sum_{j=1}^n 1/j$ for the $n$th harmonic number. For a power series
$F(z)=\sum_{t\ge0}c_tz^t$, we write $[z^t]F(z):=c_t$ for the coefficient of
$z^t$. For positive quantities $a$ and $b$ that depend on $m$ and $n$, we write
$a\sim b$ if $a/b\to1$, $a\ll b$ if $a/b\to0$, and $a\asymp b$ if $a/b$ is
bounded away from $0$ and $\infty$; the symbols $O(\cdot)$ and $o(\cdot)$ have
their usual meanings. Finally, we write $\Vol_d$ for $d$-dimensional Lebesgue
measure.

\subsection*{Normalization}
Both polytopes lie in the affine subspace $\mathcal{D}_{m,n}$ of dimension
\[
  d:=(m-1)n+m .
\]
Let $\pi$ be the coordinate projection that deletes the variables $x_{mj}\,$,
$j\in N$. Its restriction to $\mathcal{D}_{m,n}$ is an affine bijection onto $\R^d$ that
maps $(\mathbb{Z}^{m\times m}\times\mathbb{Z}^m)\cap \mathcal{D}_{m,n}$ onto
$\mathbb{Z}^d$. We define $\Vol(K):=\Vol_d(\pi(K))$ for $K\subset \mathcal{D}_{m,n}\,$.
This is the lattice-normalized relative volume; the Euclidean relative volume
equals $m^{n/2}\Vol(K)$. We abbreviate
\begin{gather*}
  \VP:=\Vol(\mathcal{P}_{m,n}\cap \mathcal{D}_{m,n}),\qquad \VQ:=\Vol(\mathcal{Q}_{m,n}\cap \mathcal{D}_{m,n}),\\
  R:=\VQ/\VP .
\end{gather*}

Let $\Delta:=\mleft\{z\in\R^m_{\ge 0}:\sum_i z_i=1\mright\}$. We take integrals over
$\Delta$ with respect to the measure that we obtain by projecting away the last
coordinate, so that $\Vol(\Delta)=1/(m-1)!$, and
\begin{equation}\label{eq:dirichlet}
  \int_\Delta z^\alpha\,dz=\frac{\prod_{i}\alpha_i!}{(m-1+|\alpha|)!}
  \qquad(\alpha\in\mathbb{Z}^m_{\ge 0}).
\end{equation}
This is the classical Dirichlet integral; equivalently, it gives the
normalizing constant of the Dirichlet distribution (see, e.g.,~\cite[Chapter~49]{KBJ}). Throughout, we set
\[
  k:=1/(m-1).
\]

\subsection*{Exact formulae}
For $y\in[0,1]^m$, we define
\begin{equation}\label{eq:f}
  f_m(y):=\frac{1}{(m-1)!}\sum_{S\subset M}(-1)^{|S|}
  \mleft(1-\sum_{i\in S}y_i\mright)_+^{m-1},
\end{equation}
and we define the power series
\begin{equation}\label{eq:g}
  g_m(z):=\sum_{t\ge 0}\frac{z^t}{t!\,(m-1+t)!}
  = z^{-(m-1)/2}\,I_{m-1}\mleft(2\sqrt{z}\mright),
\end{equation}
where $I_\nu$ is the modified Bessel function of the first kind.

\begin{theorem}[Strong relaxation]\label{thm:P}
For all $m\ge 2$ and $n\ge 1$,
\[
  \VP=\int_{[0,1]^m} f_m(y)^n\,dy .
\]
In particular, $\Vol(\mathcal{P}_{2,n}\cap \mathcal{D}_{2,n})=\dfrac{1}{(n+1)(n+2)}$.
\end{theorem}

\begin{theorem}[Weak relaxation]\label{thm:Q}
For all $m\ge 2$ and $n\ge 1$,
\[
  \VQ=\sum_{t=0}^{m}\binom{m}{t}
  \mleft(-\frac1n\mright)^{t}\,t!\,[z^{t}]\,g_m(z)^n .
\]
In particular,
\begin{gather*}
  \Vol(\mathcal{Q}_{2,n}\cap \mathcal{D}_{2,n})=\frac{3n-1}{12n},\\
  \Vol(\mathcal{Q}_{3,n}\cap \mathcal{D}_{3,n})=\frac{80n^2-15n-2}{270\cdot 2^n\,n^2}.
\end{gather*}
Moreover, $((m-1)!)^{n}\VQ$ is a polynomial in $1/n$ of degree at most $m$,
with constant term $(1-1/m)^m$.
\end{theorem}

\begin{theorem}[Weak relaxation, sign-free form]\label{thm:Qpos}
For $t\in[m]$, let
\[
  W_t:=\sum_{\sigma}\,\prod_{g\in[m]}|\sigma^{-1}(g)|!\,,
\]
where the sum is over all maps $\sigma:[t]\to[m]$ with $\sigma(i)\ne i$ for
all $i\in[t]$, and let $p_t:=(m-1)!\,W_t/(m-1+t)!$. Then, for all $m\ge2$ and
$n\ge1$,
\[
  ((m-1)!)^n\VQ=\sum_{\pi}\frac{(n)_{|\pi|}}{n^m}\prod_{B\in\pi}p_{|B|}\,,
\]
where the sum is over all partitions $\pi$ of $[m]$ into nonempty blocks,
$|\pi|$ is the number of blocks of $\pi$, and
$(n)_\ell:=n(n-1)\cdots(n-\ell+1)$. In particular, $((m-1)!)^n\VQ$ is a
polynomial in $1/n$ of degree at most $m-1$.
\end{theorem}

All terms in Theorem~\ref{thm:Qpos} are nonnegative. For example, for $m=3$,
we have $W_1=2$, $W_2=5$, and $W_3=14$, so $(p_1\,,p_2\,,p_3)=(2/3,5/12,7/30)$.
The partitions of $[3]$ contribute $(n)_3\,p_1^3$, three times
$(n)_2\,p_1p_2\,$, and $n\,p_3\,$, and summing these, we recover the formula
for $\Vol(\mathcal{Q}_{3,n}\cap \mathcal{D}_{3,n})$ in Theorem~\ref{thm:Q}.

\subsection*{Uniform asymptotics}
We define
\begin{equation}\label{eq:psi}
  \psi(n,m):=\int_0^1\mleft(1-t^{m-1}\mright)^n\,dt
  =\frac{\Gamma(1+k)\,\Gamma(n+1)}{\Gamma(n+1+k)} .
\end{equation}

\begin{theorem}\label{thm:uniform}
As $m+n\to\infty$, in any manner (that is, with no restriction on the relative
growth of $m$ and $n$),
\begin{align*}
  \VQ &\sim \mleft((m-1)!\mright)^{-n}\mleft(1-\frac1m\mright)^{m},\\
  \VP &\sim \mleft((m-1)!\mright)^{-n}\,\psi(n,m)^{m},\\
  R &\sim \prod_{j=2}^{n}\mleft(1+\frac{1}{(m-1)j}\mright)^{m}.
\end{align*}
Moreover, for all $m\ge2$ and $n\ge1$,
\begin{equation}\label{eq:bounds}
  \begin{gathered}
  \mleft(1-\frac{12}{n+1}\mright)\psi^m\le ((m-1)!)^n\VP\le \psi^m,\\
  \mleft(1-\frac{10}{m}\mright)e^{-1}\le ((m-1)!)^n\VQ\le e^{-1},
  \end{gathered}
\end{equation}
and
\begin{equation}\label{eq:logprod}
\begin{gathered}
  \ln\prod_{j=2}^{n}\mleft(1+\frac{1}{(m-1)j}\mright)^{m}
  =\frac{m}{m-1}\,(H_n-1)-\theta,\\
  0\le\theta\le\frac{1}{m-1}.
\end{gathered}
\end{equation}
\end{theorem}

In Section~\ref{sec:regimes}, we read off the behavior in each growth regime of
$m$, and we prove a transition theorem at $m\asymp\ln n$
(Theorem~\ref{thm:transition}). In Section~\ref{sec:hull}, we compare the strong
relaxation with the integer hull, and we make a detailed analysis for $m=3$, where we have the exact hull described by linear inequalities. 

\section{Exact formulae}

\begin{lemma}\label{lem:slice}
For $y\in[0,1]^m$, let
$C(y):=\{z\in\Delta: z_i\le y_i\ \forall i\in M\}$. Then
$\Vol_{m-1}(C(y))=f_m(y)$.
\end{lemma}

\begin{proof}
For $t\ge 0$, let
$F(t):=\Vol_m\{z\in\R^m:0\le z_i\le y_i\,,\ \sum_i z_i\le t\}$. For
$S\subset M$, the set $\{z\ge 0,\ \sum_i z_i\le t,\ z_i>y_i\ \forall i\in S\}$
is, after the translation $z_i\mapsto z_i-y_i$ for $i\in S$, the simplex
$\{z\ge0,\ \sum_i z_i\le t-y(S)\}$, which has volume $(t-y(S))_+^m/m!$, where
$y(S):=\sum_{i\in S}y_i\,$. Inclusion-exclusion over the constraints
$z_i\le y_i$ gives
\[
  F(t)=\frac{1}{m!}\sum_{S\subset M}(-1)^{|S|}\mleft(t-y(S)\mright)_+^{m}.
\]
Because $m\ge 2$, $F$ is continuously differentiable. The change of variables
\[
  (z_1\,,\dots,z_{m-1}\,,z_m)\mapsto(z_1\,,\dots,z_{m-1}\,,s),\qquad
  s=\sum_i z_i\,,
\]
has unit Jacobian, so by Fubini's theorem,
$F(t)=\int_0^t \phi(s)\,ds$, where $\phi(s)$ is the projected $(m-1)$-volume
of the slice $\{\sum_i z_i=s\}$ of the box. Hence,
$\Vol_{m-1}(C(y))=\phi(1)=F'(1)=f_m(y)$.
\end{proof}

\begin{proof}[Proof of Theorem~\ref{thm:P}]
In $\mathcal{P}_{m,n}\,$, the inequalities $0\le x_{ij}\le y_i$ force $y\ge 0$, so
$y\in[0,1]^m$. For fixed $y$, the constraints on $x$ decouple over the columns
$j\in N$, and each column $x_{\cdot j}$ ranges independently over $C(y)$. By
Fubini's theorem and Lemma~\ref{lem:slice}, we have
$\VP=\int_{[0,1]^m}f_m(y)^n\,dy$.

For $m=2$, from \eqref{eq:f}, we have $f_2(y)=(y_1+y_2-1)_+\,$. The variable
$u=y_1+y_2-1$ restricted to $[0,1]$ has density $1-u$ under the uniform
measure on $[0,1]^2$, so $\VP=\int_0^1u^n(1-u)\,du=\frac{1}{(n+1)(n+2)}$.
\end{proof}

\begin{remark}
For $m=2$, the formula in Theorem~\ref{thm:P} has a direct combinatorial
proof. Each column is determined by a single uniform cut $U_j\in[0,1]$, with
$x_{1j}=U_j$ and $x_{2j}=1-U_j\,$. With $y$ uniform on $[0,1]^2$, the
conditions $y_1\ge U_j$ and $y_2\ge1-U_j$ for all $j$ state that, among the
$n+2$ independent uniform variables $U_1\,,\dots,U_n\,,y_1\,,1-y_2\,$, the
variable $y_1$ is the largest and $1-y_2$ is the smallest, an event of
probability $1/((n+2)(n+1))$. For $m\ge3$, this argument breaks down: some gaps
have length equal to a difference of two cuts, and the relative order of
independent uniform variables does not determine comparisons with such
differences.
\end{remark}

\begin{proof}[Proof of Theorem~\ref{thm:Q}]
On $\mathcal{D}_{m,n}$ with $x\ge 0$, the constraints imply $x_{ij}\le 1$, so each
column satisfies $x_{\cdot j}\in\Delta$. Let $r_i:=\sum_{j\in N}x_{ij}\in[0,n]$.
Because $x\ge0$ forces $y\ge 0$, for fixed $x$, the variable $y_i$ ranges
exactly over $[r_i/n,1]$. By Fubini's theorem,
\begin{equation}\label{eq:VQint}
  \VQ=\int_{\Delta^n}\prod_{i\in M}\mleft(1-\frac{r_i}{n}\mright)\,dx .
\end{equation}
Expanding the product, we obtain
\[
  \sum_{S\subset M}(-1/n)^{|S|}\sum_{\varphi:S\to N}\,\prod_{i\in S}x_{i\varphi(i)}\,.
\]
For fixed $\varphi$, the columns are independent. Column $j$ carries the
square-free monomial $\prod_{i\in\varphi^{-1}(j)}x_{ij}\,$, whose integral is
$1/(m-1+|\varphi^{-1}(j)|)!$ by \eqref{eq:dirichlet}. Because the number of
maps $\varphi:[t]\to N$ with fiber sizes $(a_1\,,\dots,a_n)$ is
$t!/\prod_j a_j!$, the sum over $\varphi$ with $|S|=t$ equals
\[
  t!\sum_{a_1+\dots+a_n=t}\,\prod_j\frac{1}{a_j!\,(m-1+a_j)!}=t!\,[z^{t}]\,g_m^n\,.
\]
This proves the formula.

For $m=2$, we have $g_2=1+\frac z2+\frac{z^2}{12}+O(z^3)$, so
$[z]g_2^n=\frac n2$ and $[z^2]g_2^n=\frac n{12}+\binom n2\frac14$.
Substituting, we obtain
$\VQ=1-1+\frac{1}{n^2}\mleft(\frac n6+\frac{n(n-1)}4\mright)=\frac{3n-1}{12n}$.
For $m=3$, let $\tilde g:=2g_3=1+\frac z3+\frac{z^2}{24}+\frac{z^3}{360}+O(z^4)$,
so that $[z]\tilde g^n=\frac n3$, $[z^2]\tilde g^n=\frac n{24}+\binom n2\frac19$,
and $[z^3]\tilde g^n=\frac n{360}+\frac{n(n-1)}{72}+\binom n3\frac1{27}$.
Substituting into $2^n\VQ=\sum_{t=0}^3\binom3t\mleft(-\frac1n\mright)^t\,t!\,[z^t]\tilde g^n$,
we obtain $2^n\VQ=\frac{80n^2-15n-2}{270n^2}$.

For polynomiality, let $\tilde g:=(m-1)!\,g_m=1+\frac zm+O(z^2)$, and expand
$\tilde g^n=\sum_\ell\binom n\ell(\tilde g-1)^\ell$. Hence,
$t![z^{t}]\tilde g^n$ is a polynomial in $n$ of degree $t$ with leading
coefficient $m^{-t}$, and so $((m-1)!)^n\VQ$ is a polynomial in $1/n$ of
degree at most $m$, with constant term
$\sum_{t}\binom m{t}(-1/m)^{t}=(1-1/m)^m$.
\end{proof}

\begin{proof}[Proof of Theorem~\ref{thm:Qpos}]
Let the columns $x_{\cdot 1}\,,\dots,x_{\cdot n}$ be independent and uniform on
$\Delta$, and let $\E$ denote expectation. Because $\Vol(\Delta)=1/(m-1)!$,
from \eqref{eq:VQint}, we have
$((m-1)!)^n\VQ=\E\prod_{i\in M}(1-r_i/n)$. Because
$1-r_i/n=\frac1n\sum_{j\in N}(1-x_{ij})$, expanding the product, we obtain
\[
  \prod_{i\in M}\mleft(1-\frac{r_i}{n}\mright)
  =\frac{1}{n^m}\sum_{\varphi:M\to N}\,\prod_{i\in M}\mleft(1-x_{i\varphi(i)}\mright).
\]
For fixed $\varphi$, the columns are independent, so the expectation of the
summand is
\[
  \prod_{j\in N}\E\prod_{i\in\varphi^{-1}(j)}(1-x_{ij}).
\]
The
uniform distribution on $\Delta$ is invariant under permutations of the
coordinates, so, for $z$ uniform on $\Delta$ and $I\subset M$, the quantity
$\E\prod_{i\in I}(1-z_i)$ depends only on $|I|$. We claim that it equals
$p_{|I|}\,$. Indeed, for $I=[t]$, we have
\[
  \prod_{i\in[t]}(1-z_i)=\prod_{i\in[t]}\,\sum_{g\ne i}z_g
  =\sum_{\sigma}z^{\alpha(\sigma)},
\]
where $\sigma$ ranges as in the definition of $W_t$ and
$\alpha(\sigma)_g:=|\sigma^{-1}(g)|$, and from \eqref{eq:dirichlet}, we have
$\E z^{\alpha(\sigma)}=(m-1)!\prod_g\alpha(\sigma)_g!/(m-1+t)!$. Finally, the
nonempty fibers of $\varphi$ form a partition $\pi$ of $M$, and each partition
with $\ell$ blocks arises from exactly $(n)_\ell$ maps $\varphi$, which proves
the formula. Because $(n)_\ell$ is a polynomial in $n$ of degree $\ell\le m$
with zero constant term, $(n)_\ell/n^m$ is a polynomial in $1/n$ of degree at
most $m-1$.
\end{proof}

\begin{remark}
The quantities in Theorem~\ref{thm:Qpos} have a direct probabilistic interpretation.
We drop $m-1$ independent uniform points (``cuts'') into $[0,1]$; the lengths of
the $m$ resulting gaps, from left to right, form a uniform point of $\Delta$.
We then drop $t$ further independent uniform points, labeled $1,\dots,t$.
Now, $p_t$ is the probability that, for each $i\in[t]$, point~$i$ does not lie
in the $i$th gap. Equivalently, $W_t$ counts the words in $m-1$ copies of a
letter $b$ and the distinct letters $1,\dots,t$ in which letter~$i$ is not
preceded by exactly $i-1$ copies of $b$; there are $(m-1+t)!/(m-1)!$ such
words in all. In this language, $((m-1)!)^n\VQ$ is the probability that $m$
test points, each placed in an independent, uniformly random column, all avoid
their forbidden gaps. Because $W_1=m-1$, we have $p_1=1-1/m$, and, as $n\to\infty$,
the term with $|\pi|=m$ tends to $(1-1/m)^m$, the constant term in
Theorem~\ref{thm:Q}.
\end{remark}

\section{Uniform asymptotics}\label{sec:uniform}

In the proofs in this section and in Section~\ref{sec:regimes}, we use some
standard probabilistic language, which we recall. We write $\Pr$ and $\E$ for
probability and expectation, and $\Unif(K)$ for the uniform distribution on a
set $K$. We write $\mathrm{Exp}(1)$ for the exponential distribution with
density $e^{-t}$ on $[0,\infty)$, and $\mathrm{Beta}(\alpha,\beta)$ for the
distribution on $[0,1]$ with density proportional to
$t^{\alpha-1}(1-t)^{\beta-1}$. For random variables $X_\ell$ indexed by a
growing parameter $\ell$, we say that $X_\ell\to c$ \emph{in probability} if
$\Pr(|X_\ell-c|>\eta)\to0$ for every $\eta>0$, and we write
$X_\ell=O_p(a_\ell)$ if, for every $\eta>0$, there is a constant $C$ with
$\Pr(|X_\ell|>Ca_\ell)<\eta$ for all large $\ell$. We say that $X_\ell$
converges \emph{in distribution} to $X$, and we write $X_\ell\Rightarrow X$, if
$\Pr(X_\ell\le t)\to\Pr(X\le t)$ at every continuity point $t$ of the limit.
For probability measures $\nu$ and $\mu$ with densities $p$ and $q$, we use the
total-variation distance $\|\nu-\mu\|_{\mathrm{TV}}:=\int|p-q|$; if it tends
to $0$, then $|\nu(A)-\mu(A)|\to0$ uniformly over events $A$. Finally, we use
the bounded convergence theorem to pass limits through integrals of uniformly
bounded functions; in particular, if $X_\ell\to c$ in probability and
$|X_\ell|\le1$, then $\E X_\ell\to c$.

\subsection*{Probabilistic reformulation}
Let the columns $x_{\cdot 1}\,,\dots,x_{\cdot n}$ be independent and uniform on
$\Delta$. Let $s_i:=r_i/n$ and $M_i:=\max_{j}x_{ij}\,$. Integrating out $y$ in
$\mathcal{P}_{m,n}\cap \mathcal{D}_{m,n}$ (where $y_i\in[M_i\,,1]$) and using \eqref{eq:VQint}, we
obtain
\begin{equation}\label{eq:prob}
  \begin{gathered}
  \hat V_P:=((m-1)!)^n\VP=\E\prod_{i\in M}(1-M_i),\\
  \hat V_Q:=((m-1)!)^n\VQ=\E\prod_{i\in M}(1-s_i).
  \end{gathered}
\end{equation}
The two relaxations differ only in replacing the row \emph{maximum} by the row
\emph{average}. Here, probabilities and expectations are simply normalized
volumes and integrals over $\Delta^n$, so \eqref{eq:prob} is the statement that
these volumes are averages over the $n$ columns. The marginal law of each $x_{ij}$ is
$\mathrm{Beta}(1,m-1)$, so $\Pr(M_i\le t)=(1-(1-t)^{m-1})^n$ and
$\E(1-M_i)=\psi(n,m)$. The substitution $w=t^{m-1}$ gives $\psi=k\,B(k,n+1)$,
which is the Gamma form in \eqref{eq:psi}. Thus, from
Theorem~\ref{thm:uniform}, we have $\hat V_P\sim\prod_i\E(1-M_i)$ and
$\hat V_Q\sim\prod_i\E(1-s_i)$: asymptotically, the rows behave as if they
were independent.

\subsection*{The strong relaxation}
Let $h:=(m-1)!\,f_m\,$, so that $h(y)=\Pr_{z\sim\Unif(\Delta)}(z\le y)$ by
Lemma~\ref{lem:slice}, and $\hat V_P=\int_{[0,1]^m}h^n$. Let
$h_i(y_i):=\Pr(z_i\le y_i)=1-(1-y_i)^{m-1}$.

\begin{lemma}[Negative orthant dependence]\label{lem:NOD}
For all $y\in[0,1]^m$, we have $h(y)\le\prod_{i\in M}h_i(y_i)$.
\end{lemma}

In words, capping one coordinate of a point of $\Delta$ leaves more mass for
the other coordinates, so the caps are harder to satisfy together than
separately.

\begin{proof}
We show that $\Pr(z_i\le y_i\ \forall i\in I)\le\prod_{i\in I}\Pr(z_i\le y_i)$
by induction on $|I|$; the case $|I|\le1$ is trivial. We fix $i_0\in I$, and we
let $I':=I\setminus\{i_0\}$. Conditionally on $z_{i_0}=t<1$, the vector
$(z_i)_{i\ne i_0}/(1-t)$ is uniform on the simplex of dimension $m-2$. Hence,
the conditional probability
$\rho(t):=\Pr(z_i\le y_i\ \forall i\in I'\mid z_{i_0}=t)$ is the probability
of the event $\{w_i\le y_i/(1-t)\ \forall i\in I'\}$ for a fixed uniform vector
$w$, and so $\rho$ is nondecreasing in $t$. Because $\one\{t\le y_{i_0}\}$ is
nonincreasing in $t$, we can apply Chebyshev's association inequality (for a
real random variable $T$, a nonincreasing function $u$, and a nondecreasing
function $v$, we have $\E\mleft(u(T)v(T)\mright)\le\E\mleft(u(T)\mright)\,\E\mleft(v(T)\mright)$), and we obtain
$\E\mleft(\one\{z_{i_0}\le y_{i_0}\}\rho(z_{i_0})\mright)\le\Pr(z_{i_0}\le y_{i_0})\,\E\rho(z_{i_0})$.
The right-hand side equals
$\Pr(z_{i_0}\le y_{i_0})\Pr(z_i\le y_i\ \forall i\in I')$, and we complete the
proof by applying the induction hypothesis.
\end{proof}

\begin{proposition}\label{prop:Pbound}
For all $m\ge 2$ and $n\ge 1$, we have
\[
  \mleft(1-\frac{12}{n+1}\mright)\psi^m\le\hat V_P\le\psi^m .
\]
\end{proposition}

\begin{proof}
The upper bound follows from Lemma~\ref{lem:NOD}, because
$\int_{[0,1]^m}\prod_ih_i^n=\prod_i\int_0^1h_i^n=\psi^m$.

For the lower bound, let $\mu$ be the probability measure on $[0,1]^m$ with
density $\prod_ih_i(y_i)^n/\psi^m$. Under $\mu$, the variables
$W_i:=(1-y_i)^{m-1}$ are independent. Substituting $w=\varepsilon^{m-1}$ in
the density $\propto(1-\varepsilon^{m-1})^n$ of $\varepsilon=1-y_i\,$, we see
that each $W_i$ has law $\mathrm{Beta}(k,n+1)$. Hence,
\[
  \E W_i=\frac{k}{n+1+k},\qquad \E W_i^2=\frac{k(k+1)}{(n+1+k)(n+2+k)} .
\]
Let $T:=\sum_iW_i\,$. Because $mk\le 2$ and $k\le1$, we have
$\E T\le 2/(n+1)$ and $\E T^2\le m\,\E W_1^2+(\E T)^2\le 8/(n+1)^2$.

By the union bound, $h\ge 1-T$. Also, $\prod_ih_i=\prod_i(1-W_i)\le e^{-T}$,
and $\ln(1-x)\ge -x-x^2$ on $\mleft[0,\frac12\mright]$. So, on the event $\mleft\{T\le\frac12\mright\}$,
we have
\[
  h^n\ge(1-T)^n\ge e^{-nT}e^{-nT^2}\ge e^{-nT^2}\prod_ih_i^n .
\]
Therefore,
\begin{align*}
  \hat V_P&\ge\psi^m\,\E_\mu\mleft(e^{-nT^2}\one\mleft\{T\le\tfrac12\mright\}\mright)\\
  &\ge\psi^m\mleft(1-\Pr_\mu\mleft(T>\tfrac12\mright)-n\E_\mu T^2\mright)\\
  &\ge\psi^m\mleft(1-\frac{4}{n+1}-\frac{8n}{(n+1)^2}\mright),
\end{align*}
where we use $e^{-x}\ge1-x$ and Markov's inequality.
\end{proof}

\begin{proposition}\label{prop:fixedn}
For fixed $n$, as $m\to\infty$, both $\hat V_P$ and $\psi(n,m)^m$ tend to
$e^{-H_n}$.
\end{proposition}

\begin{proof}
From \eqref{eq:psi}, we have
$\ln\psi=\ln\Gamma(1+k)-(\ln\Gamma(n+1+k)-\ln\Gamma(n+1))
=-\gamma k-k(H_n-\gamma)+O(k^2)$. Hence, $\psi^m\to e^{-H_n}$.

For $\hat V_P\,$, let $x_{ij}=E_{ij}/S_j\,$, where the $E_{ij}$ are
independent and identically distributed with law $\mathrm{Exp}(1)$, and
$S_j:=\sum_iE_{ij}\,$. Then
\[
  \frac{1}{\max_jS_j}\sum_i\max_jE_{ij}\ \le\ \sum_iM_i\ \le\
  \frac{1}{\min_jS_j}\sum_i\max_jE_{ij}\,.
\]
By the law of large numbers, $S_j/m\to1$ for each of the finitely many $j$,
and $\frac1m\sum_i\max_jE_{ij}\to\E\max_{j\le n}E_{1j}=H_n\,$. So,
$\sum_iM_i\to H_n$ in probability. Also,
\[
  \max_iM_i\le\frac{\max_{i,j}E_{ij}}{\min_jS_j}=O_p\mleft(\frac{\ln m}{m}\mright)\to0,
\]
so $\sum_iM_i^2\le\max_iM_i\sum_iM_i\to0$. On the event
$\mleft\{\max_iM_i\le\frac12\mright\}$, we have
$e^{-\sum M_i-\sum M_i^2}\le\prod_i(1-M_i)\le e^{-\sum M_i}$. Hence,
$\prod_i(1-M_i)\to e^{-H_n}$ in probability, and by bounded convergence,
$\hat V_P\to e^{-H_n}$.
\end{proof}

\subsection*{The weak relaxation}

\begin{proposition}\label{prop:Qbound}
For all $m\ge2$ and $n\ge1$, we have
\[
  \mleft(1-\frac{10}{m}\mright)e^{-1}\le\hat V_Q\le e^{-1}.
\]
\end{proposition}

\begin{proof}
Because $\sum_is_i=1$, we have $\prod_i(1-s_i)\le e^{-1}$. For the lower
bound, we note that each $x_{ij}$ has variance $\frac{m-1}{m^2(m+1)}$, so
$X:=\sum_is_i^2$ satisfies $\E X=\frac1m+\frac{m-1}{nm(m+1)}\le\frac2m$.
Moreover, $\Pr\mleft(\max_is_i>\frac12\mright)\le4\E X$. On the event
$\mleft\{\max_is_i\le\frac12\mright\}$, we have $\prod_i(1-s_i)\ge e^{-1-X}\ge e^{-1}(1-X)$.
Therefore, $\hat V_Q\ge e^{-1}(1-4\E X-\E X)\ge e^{-1}(1-10/m)$.
\end{proof}

\begin{proof}[Proof of Theorem~\ref{thm:uniform}]
The bounds \eqref{eq:bounds} are Propositions~\ref{prop:Pbound}
and~\ref{prop:Qbound}. For the asymptotics, it suffices to show that every
sequence with $m+n\to\infty$ has a subsequence along which the claimed ratios
tend to $1$.

\emph{Strong relaxation.} If $n\to\infty$ along a subsequence, then we apply
Proposition~\ref{prop:Pbound}. Otherwise, we pass to a subsequence with $n$
constant; then $m\to\infty$, and we apply Proposition~\ref{prop:fixedn}, which
suffices because the limit $e^{-H_n}$ is positive.

\emph{Weak relaxation.} If $m\to\infty$ along a subsequence, then we apply
Proposition~\ref{prop:Qbound} together with
$e^{-1}\mleft(1-\frac1m\mright)\le\mleft(1-\frac1m\mright)^m\le e^{-1}$. Otherwise, we pass to a
subsequence with $m$ constant; then $n\to\infty$, and we apply
Theorem~\ref{thm:Q}.

\emph{Ratio.} Dividing the two asymptotics, we obtain
$R\sim\mleft((1+k)\psi\mright)^{-m}$, because $1-\frac1m=\frac{1}{1+k}$. Now,
\[
  (1+k)\psi=\frac{\Gamma(2+k)\Gamma(n+1)}{\Gamma(n+1+k)}
  =\prod_{j=2}^n\frac{j}{j+k} .
\]
Finally, \eqref{eq:logprod} follows from $0\le x-\ln(1+x)\le x^2/2$:
\[
  0\le\theta=m\sum_{j=2}^n\mleft(\frac kj-\ln\mleft(1+\tfrac kj\mright)\mright)
  \le\frac{mk^2}{2}\mleft(\frac{\pi^2}{6}-1\mright)\le\frac{1}{m-1}.
\]
Here, we use $\frac{m}{m-1}\sum_{j=2}^n\frac1j=\frac{m}{m-1}(H_n-1)$ and
$mk\le 2$.
\end{proof}

\begin{remark}
For $n=1$, the two relaxations coincide, and indeed, the product in
Theorem~\ref{thm:uniform} is empty. For $m=2$, from Theorem~\ref{thm:uniform},
we have $R\sim\frac{(n+1)^2}{4}$, in agreement with
Theorems~\ref{thm:P} and~\ref{thm:Q}.
\end{remark}

\section{Asymptotic regimes and the transition at \texorpdfstring{$m\asymp\ln n$}{m ~ ln n}}
\label{sec:regimes}

From Theorem~\ref{thm:uniform} and \eqref{eq:logprod}, we have, uniformly,
\begin{equation}\label{eq:master}
  \begin{split}
  \ln R&=\frac{m}{m-1}(H_n-1)+O\mleft(\frac1m\mright)+o(1)\\
  &=\ln n+\frac{\ln n}{m-1}+\frac{m}{m-1}(\gamma-1)+O\mleft(\frac1m\mright)+o(1),
  \end{split}
\end{equation}
where we assume $n\to\infty$ for the second form. For fixed $m$, the exact
product is more precise:
\[
  R\sim\mleft(\frac{\Gamma(n+1+k)}{\Gamma(2+k)\Gamma(n+1)}\mright)^m
  \sim\frac{n^{m/(m-1)}}{\Gamma\mleft(\frac{2m-1}{m-1}\mright)^m}.
\]
The decisive term in \eqref{eq:master} is $\frac{\ln n}{m-1}$, which is
unbounded, bounded, or negligible according as $m\ll\ln n$, $m\asymp\ln n$, or
$m\gg\ln n$.

\begin{corollary}\label{cor:regimes}
As $n\to\infty$, $R$ behaves as follows.
\begin{center}
\renewcommand{\arraystretch}{1.6}
\begin{tabular}{@{}ll@{}}
\toprule
growth of $m$ & $R$\\
\midrule
$m$ fixed
  & $\sim\dfrac{n^{m/(m-1)}}{\Gamma\mleft(\frac{2m-1}{m-1}\mright)^m}$\\
$m\sim\beta\ln n$
  & $\sim e^{\gamma-1+1/\beta}\,n$\\
$m/\ln n\to\infty$
  & $\sim e^{\gamma-1}\,n$\\
\bottomrule
\end{tabular}
\end{center}
Here, $e^{\gamma-1}\approx0.655$.
\end{corollary}

The only remaining case is bounded $n$ with $m\to\infty$. There,
$R\to e^{H_n-1}$ stays bounded; for $n=1$, the relaxations coincide.

\subsection*{The transition}
From the table, we see that $R$ is always of order at least $n$ (for
$n\to\infty$), and that the only question is the size of the excess factor
$R/n$. This factor is unbounded for $m\ll\ln n$, bounded for
$m\asymp\ln n$, and tends to the constant $e^{\gamma-1}$ for $m\gg\ln n$. In
the next theorem, we make the transition precise, and we relate it to the
geometry of a typical point of $\mathcal{P}_{m,n}\cap \mathcal{D}_{m,n}\,$.

\begin{theorem}[Transition at $m\asymp\ln n$]\label{thm:transition}
Let $m,n\to\infty$ with $\frac{\ln n}{m-1}\to\tau\in[0,\infty]$. Then the
following hold.
\begin{enumerate}
\item We have $R/n\to e^{\gamma-1+\tau}$, where we interpret the limit
as $+\infty$ when $\tau=\infty$.
\item If $(x,y)$ is uniform on $\mathcal{P}_{m,n}\cap \mathcal{D}_{m,n}\,$, then each $y_i$
converges in distribution to $\Unif[1-e^{-\tau},1]$. This limit is the point
mass at $1$ when $\tau=\infty$, and it is $\Unif[0,1]$ when $\tau=0$.
\item If $(x,y)$ is uniform on $\mathcal{Q}_{m,n}\cap \mathcal{D}_{m,n}\,$, then each $y_i$
converges in distribution to $\Unif[0,1]$, for every $\tau$.
\end{enumerate}
In particular, the critical scaling is $m\sim\beta\ln n$, with $\tau=1/\beta$.
\end{theorem}

\begin{proof}
(1) For $\tau<\infty$, this follows from \eqref{eq:master} together with
$\frac{m}{m-1}(\gamma-1)\to\gamma-1$. For $\tau=\infty$, from
\eqref{eq:logprod}, we have
$\ln R-\ln n\ge(H_n-1-\ln n)+\frac{H_n-1}{m-1}-1+o(1)\to\infty$.

(2) By the proof of Theorem~\ref{thm:P}, the $y$-marginal $\nu$ of the uniform
measure on $\mathcal{P}_{m,n}\cap \mathcal{D}_{m,n}$ has density $h^n/\hat V_P\,$. Let $\mu$ be the
product measure from the proof of Proposition~\ref{prop:Pbound}. Because
$h^n\le\prod_ih_i^n$ by Lemma~\ref{lem:NOD}, a direct computation gives
$\|\nu-\mu\|_{\mathrm{TV}}\le 2(1-\hat V_P/\psi^m)\le 24/(n+1)\to0$.

Under $\mu$, the variable $\varepsilon=1-y_i$ has density proportional to
$(1-\varepsilon^{m-1})^n$. We write
$n\varepsilon^{m-1}=\exp\mleft((m-1)\mleft(\ln\varepsilon+\frac{\ln n}{m-1}\mright)\mright)$.
Suppose first that $\tau<\infty$. Then $n\varepsilon^{m-1}\to0$ for
$\varepsilon<e^{-\tau}$, so the density tends to $1$ there, by
$(1-w)^n\ge1-nw$. Also, $n\varepsilon^{m-1}\to\infty$ for
$\varepsilon>e^{-\tau}$, so the density tends to $0$ there, by
$(1-w)^n\le e^{-nw}$. By bounded convergence, we then have
$\varepsilon\Rightarrow\Unif[0,e^{-\tau}]$.

If $\tau=\infty$, then we fix $\eta>0$, and we let $a:=n^{-1/(m-1)}$. The mass
on $(\eta,1]$ is at most $\exp(-n\eta^{m-1})\le e^{-\sqrt n}$ once
$\frac{\ln n}{m-1}\ge2\ln\frac1\eta$. The mass on $[0,a]$ is at least $a/4$,
and $\ln(1/a)\le\ln n$. So, $\Pr(\varepsilon>\eta)\to0$.

(3) By \eqref{eq:VQint}, the $x$-marginal of the uniform measure on
$\mathcal{Q}_{m,n}\cap \mathcal{D}_{m,n}$ has density $\prod_i(1-s_i)/\hat V_Q\le\mleft(1-\frac{10}m\mright)^{-1}$
with respect to the product of uniform measures on $\Delta^n$. Under the
latter, $\E s_i=\frac1m\to0$, so $s_i\to0$ in probability under the uniform
measure on $\mathcal{Q}_{m,n}\cap \mathcal{D}_{m,n}$ as well. Given $x$, the variable $y_i$ is
uniform on $[s_i\,,1]$, which proves (3).
\end{proof}

\section{The integer hull}\label{sec:hull}

We compare the strong relaxation with the integer hull
\[
  \mathcal{H}_{m,n}:=\operatorname{conv}\mleft\{(x,y)\in \mathcal{P}_{m,n}\cap \mathcal{D}_{m,n}:\ y\in\{0,1\}^m\mright\}.
\]
Because $\mathcal{P}_{m,n}$ and $\mathcal{Q}_{m,n}$ contain the same points with $y\in\{0,1\}^m$,
the polytope $\mathcal{H}_{m,n}$ is also the integer hull of $\mathcal{Q}_{m,n}\cap \mathcal{D}_{m,n}\,$,
and $\mathcal{H}_{m,n}\subset \mathcal{P}_{m,n}\cap \mathcal{D}_{m,n}\subset \mathcal{Q}_{m,n}\cap \mathcal{D}_{m,n}\,$.

\begin{theorem}[Cho, Johnson, Padberg, and Rao~\cite{CJPR83}]\label{thm:hull}
For all $n\ge1$, the polytope $\mathcal{H}_{3,n}$ is the set of points
$(x,y)\in \mathcal{P}_{3,n}\cap \mathcal{D}_{3,n}$ that satisfy
\[
  x_{1a}+x_{2b}+x_{3c}+y_1+y_2+y_3\ge2\qquad\forall a,b,c\in N.
\]
The inequalities in which $a$, $b$, and $c$ are not all distinct are implied by
$\mathcal{P}_{3,n}\cap \mathcal{D}_{3,n}\,$; in particular, $\mathcal{H}_{3,n}=\mathcal{P}_{3,n}\cap \mathcal{D}_{3,n}$ for $n\le2$.
\end{theorem}

Cho, Johnson, Padberg, and Rao~\cite[Theorem~6.1 and Remark~2.3]{CJPR83}
describe all facets of the integer polytope of the set-packing formulation in
the variables $x$ and $1-y$, in which the demand constraints are relaxed to
$\sum_ix_{ij}\le1$. Because $\mathcal{H}_{3,n}$ is the face of that polytope on which the
demand constraints hold with equality, their description yields
Theorem~\ref{thm:hull}. Using the demand constraints, we can rewrite the
inequalities of Theorem~\ref{thm:hull} as the odd-cycle
inequalities~\cite{Pad73} of the chordless cycles of length $9$ in the
intersection graph of that formulation~\cite{CJPR83}. Cho, Padberg, and
Rao~\cite[Section~5]{CPR83} also describe all facets for $m\ge3$ facilities and
$n=3$ customers. We include a short proof of Theorem~\ref{thm:hull}, different
from that of~\cite{CJPR83}, because we use its construction in
Lemma~\ref{lem:hullsuff} below.

\begin{proof}
For a nonempty set $S\subset M$, let $\chi_S\in\{0,1\}^M$ be its characteristic
vector, and let $\Delta_S:=\{z\in\Delta: z_i=0\ \forall i\notin S\}$. The
polytope $\mathcal{H}_{3,n}$ is the convex hull of the union of the polytopes
$\{\chi_S\}\times\Delta_S^n\,$. So $(x,y)\in \mathcal{H}_{3,n}$ if and only if there are
weights $\lambda_S\ge0$ with $\sum_S\lambda_S=1$ and $y=\sum_S\lambda_S\chi_S$
such that every column $x_{\cdot j}$ lies in the Minkowski sum
$\sum_S\lambda_S\Delta_S\,$. This Minkowski sum is the base polytope of the
submodular function $T\mapsto\sum_{S\cap T\ne\emptyset}\lambda_S\,$. Because
$m=3$, it consists of the $z\in\Delta$ with $z_i\le y_i$ and
$z_i\ge\mu_i:=\lambda_{\{i\}}$ for all $i\in M$. Solving for the weights of
the sets with two or three elements, we see that a vector $\mu$ arises from
some admissible $\lambda$ if and only if $0\le\mu_i\le y_i$ and
$y_i+\sum_{l\ne i}\mu_l\le1$ for all $i$, and $\sum_i\mu_i\ge2-\sum_iy_i\,$.

Let $m_l:=\min_jx_{lj}\,$. It follows that a point
$(x,y)\in \mathcal{P}_{3,n}\cap \mathcal{D}_{3,n}$ lies in $\mathcal{H}_{3,n}$ if and only if the maximum of
$\sum_l\mu_l$ over the $\mu$ with $0\le\mu\le m$ and
$\sum_{l\ne i}\mu_l\le1-y_i$ for all $i$ is at least $2-\sum_iy_i\,$. By
linear-programming duality, and because the dual feasible region has eight
vertices, this maximum equals the least of $\sum_lm_l\,$, of $m_i+1-y_i$ for
$i\in M$, of $2-y_i-y_k$ for $i\ne k$, and of $(3-\sum_ly_l)/2$. On
$\mathcal{P}_{3,n}\cap \mathcal{D}_{3,n}\,$, each of these quantities except the first is at least
$2-\sum_iy_i\,$, because $x_{ij}\ge1-y_k-y_l$ for $\{i,k,l\}=M$, because
$y\ge0$, and because $\sum_ly_l\ge1$. Hence $(x,y)\in \mathcal{H}_{3,n}$ if and only if
$\sum_lm_l+\sum_ly_l\ge2$, which is the displayed family.

Finally, suppose that $a=b$, say. Because $y_1+y_2\ge1-x_{3c}$ and
$x_{3a}\le y_3\,$, we have
$x_{1a}+x_{2a}+x_{3c}+\sum_ly_l=1-x_{3a}+x_{3c}+\sum_ly_l\ge2+y_3-x_{3a}\ge2$.
The other cases are symmetric.
\end{proof}

The next lemma gives a simple sufficient condition for membership in the
integer hull, valid for all $m$.

\begin{lemma}\label{lem:hullsuff}
Let $m\ge2$. If $(x,y)\in \mathcal{P}_{m,n}\cap \mathcal{D}_{m,n}$ and $\sum_ly_l\ge m-1$, then
$(x,y)\in \mathcal{H}_{m,n}\,$.
\end{lemma}

\begin{proof}
Let $\varepsilon_l:=1-y_l\,$, so that $\sum_l\varepsilon_l\le1$. We put the
weight $\lambda_M:=1-\sum_l\varepsilon_l$ on $M$, and the weight
$\lambda_{M\setminus\{l\}}:=\varepsilon_l$ on $M\setminus\{l\}$ for each
$l\in M$. These weights are nonnegative, they sum to $1$, and they satisfy
$\sum_{S\ni i}\lambda_S=y_i$ for all $i\in M$. As in the proof of
Theorem~\ref{thm:hull}, the Minkowski sum $\sum_S\lambda_S\Delta_S$ is the base
polytope of $T\mapsto\sum_{S\cap T\ne\emptyset}\lambda_S\,$. With
$U:=M\setminus T$ and $z(U):=\sum_{i\in U}z_i\,$, it consists of the
$z\in\Delta$ with $z(U)\ge\sum_{S\subset U}\lambda_S$ for every proper subset
$U$ of $M$. The right-hand side vanishes unless $U=M\setminus\{l\}$ for some
$l$, and in that case the constraint reads $1-z_l\ge\varepsilon_l\,$, that is,
$z_l\le y_l\,$. Hence every column of $x$ lies in the Minkowski sum, and
$(x,y)\in \mathcal{H}_{m,n}\,$.
\end{proof}

For $m=2$, every point of $\mathcal{P}_{2,n}\cap \mathcal{D}_{2,n}$ satisfies $y_1+y_2\ge1$, so
Lemma~\ref{lem:hullsuff} gives $\mathcal{H}_{2,n}=\mathcal{P}_{2,n}\cap \mathcal{D}_{2,n}\,$, which also
follows from total unimodularity~\cite[Remark~2.3]{CJPR83}. For $m=3$, the
condition of Lemma~\ref{lem:hullsuff} implies the inequalities of
Theorem~\ref{thm:hull}.

\begin{proposition}\label{prop:hullclose}
Let $m\ge2$ and $n\ge1$. Then
\[
  0\le\Vol(\mathcal{P}_{m,n}\cap \mathcal{D}_{m,n})-\Vol(\mathcal{H}_{m,n})
  \le ((m-1)!)^{-n}\,e^{-n\,m^{2-m}} .
\]
Moreover, if $n>11$, then
\[
  1-\frac{\Vol(\mathcal{H}_{m,n})}{\Vol(\mathcal{P}_{m,n}\cap \mathcal{D}_{m,n})}
  \le\frac{\mu(\varepsilon_1+\dots+\varepsilon_m>1)}{1-12/(n+1)}\,,
\]
where $\mu$ is the probability measure on $[0,1]^m$ under which the
coordinates $\varepsilon_l$ are independent, each with density proportional to
$(1-\varepsilon^{m-1})^n$.
\end{proposition}

\begin{proof}
By Lemma~\ref{lem:hullsuff}, every point of $\mathcal{P}_{m,n}\cap \mathcal{D}_{m,n}$ outside
$\mathcal{H}_{m,n}$ satisfies $\sum_l\varepsilon_l>1$, where $\varepsilon_l:=1-y_l\,$.
For fixed $y$, the $x$-fiber of $\mathcal{P}_{m,n}\cap \mathcal{D}_{m,n}$ has volume $f_m(y)^n$,
by the proof of Theorem~\ref{thm:P}, and $(m-1)!\,f_m(y)=h(y)$, with $h$ and
$h_l$ as in Section~\ref{sec:uniform}. Hence
\begin{multline*}
  ((m-1)!)^n\mleft(\Vol(\mathcal{P}_{m,n}\cap \mathcal{D}_{m,n})-\Vol(\mathcal{H}_{m,n})\mright)\\
  \le\int_{\sum_l\varepsilon_l>1}h(y)^n\,d\varepsilon
  \le\int_{\sum_l\varepsilon_l>1}\prod_lh_l(y_l)^n\,d\varepsilon,
\end{multline*}
where the integrals are over $[0,1]^m$ and the second inequality is
Lemma~\ref{lem:NOD}. Because $h_l(y_l)=1-\varepsilon_l^{m-1}$, the
last integral equals $\psi(n,m)^m\,\mu(\varepsilon_1+\dots+\varepsilon_m>1)$,
and the second statement follows from Proposition~\ref{prop:Pbound}. For the
first statement, we use
$\prod_l(1-\varepsilon_l^{m-1})^n\le e^{-n\sum_l\varepsilon_l^{m-1}}$ together
with the bound
$\sum_l\varepsilon_l^{m-1}\ge m^{2-m}(\sum_l\varepsilon_l)^{m-1}\ge m^{2-m}$
on the region of integration, which follows from the convexity of
$t\mapsto t^{m-1}$; the region has volume at most $1$.
\end{proof}

\begin{corollary}\label{cor:hullclose}
The following hold.
\begin{enumerate}
\item For fixed $m\ge2$,
$1-\Vol(\mathcal{H}_{m,n})/\Vol(\mathcal{P}_{m,n}\cap \mathcal{D}_{m,n})
=O\mleft(n^{m/(m-1)}e^{-n\,m^{2-m}}\mright)$. In particular, for $m=3$, we have
$0\le\Vol(\mathcal{P}_{3,n}\cap \mathcal{D}_{3,n})-\Vol(\mathcal{H}_{3,n})\le2^{-n}e^{-n/3}$, and the relative
difference is $O\mleft(n^{3/2}e^{-n/3}\mright)$; Proposition~\ref{prop:m3sharp}
sharpens this to the exact rate.
\item If $m,n\to\infty$ with $\frac{\ln n}{m-1}-\ln m\to\infty$, then
$\Vol(\mathcal{H}_{m,n})\sim\Vol(\mathcal{P}_{m,n}\cap \mathcal{D}_{m,n})$.
\end{enumerate}
\end{corollary}

\begin{proof}
(1) This follows from the first bound of Proposition~\ref{prop:hullclose} and
Proposition~\ref{prop:Pbound}, because $\psi(n,m)^m\asymp n^{-m/(m-1)}$ for
fixed $m$, and because $3^{2-3}=1/3$.

(2) Let $k:=1/(m-1)$. By Markov's inequality,
$\mu(\varepsilon_1+\dots+\varepsilon_m>1)\le m\,\E_\mu\varepsilon_1\,$. As in
the proof of Proposition~\ref{prop:Pbound}, $\varepsilon_1^{m-1}$ has law
$\mathrm{Beta}(k,n+1)$ under $\mu$, so
\[
  \E_\mu\varepsilon_1
  =\frac{\Gamma(2k)}{\Gamma(k)}\cdot\frac{\Gamma(n+1+k)}{\Gamma(n+1+2k)}
  \le3\,n^{-k},
\]
because $\Gamma(2k)/\Gamma(k)=\Gamma(2k+1)/(2\Gamma(k+1))\le2/(2\cdot0.88)$
for $0<k\le1$, and because the log-convexity of $\Gamma$ gives
$\Gamma(u)/\Gamma(u+k)\le(u+k)^{1-k}/u\le2u^{-k}$ for $u\ge1$. Hence
$\mu(\varepsilon_1+\dots+\varepsilon_m>1)\le3\exp\mleft(\ln m-\frac{\ln n}{m-1}\mright)\to0$,
and the claim follows from Proposition~\ref{prop:hullclose}.
\end{proof}

The condition in part~(2) of Corollary~\ref{cor:hullclose} holds, for example,
when $m\le(1-\delta)\ln n/\ln\ln n$ for a fixed $\delta>0$. It lies inside the
regime $m\ll\ln n$ of Theorem~\ref{thm:transition}, where the facility
variables of a typical point of $\mathcal{P}_{m,n}\cap \mathcal{D}_{m,n}$ are close to $1$, which
is exactly where Lemma~\ref{lem:hullsuff} applies. For $m\gg\ln n$, the
condition $\sum_ly_l\ge m-1$ fails for typical points, and we do not know
whether $\mathcal{H}_{m,n}$ remains volumetrically close to $\mathcal{P}_{m,n}\cap \mathcal{D}_{m,n}\,$.

\begin{figure*}[t]
\centering
\includegraphics[width=\textwidth]{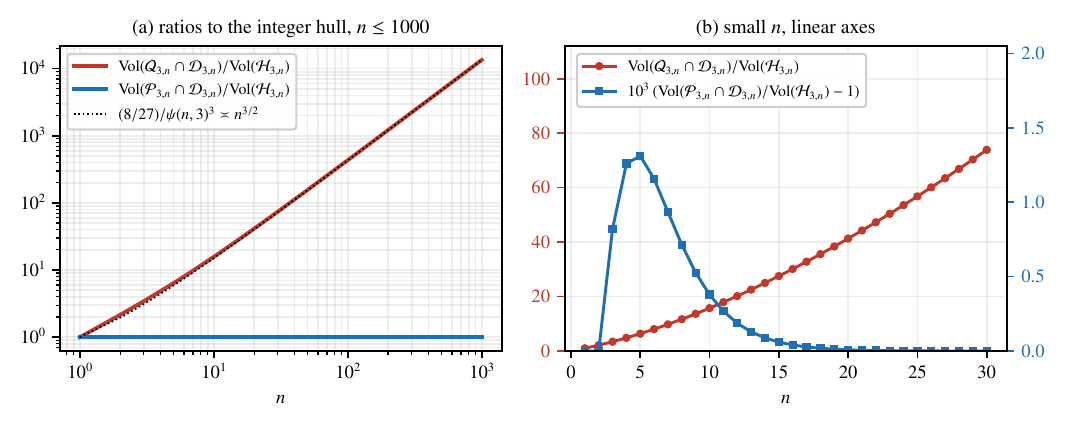}
\caption{Three facilities: the weak and the strong relaxation, compared with
the integer hull $\mathcal{H}_{3,n}\,$.}
\label{fig:m3}
\end{figure*}

We now give exact volumes. Let $S_d:=\{v\in\R^d_{\ge0}:\sum_iv_i\le1\}$, with
Lebesgue measure; let $\psi_n:=\psi(n,3)=2^n\,n!/(2n+1)!!$; and let
$(n)_3:=n(n-1)(n-2)$.

\begin{theorem}[Volumes for three facilities]\label{thm:m3vol}
For all $n\ge1$,
\begin{multline*}
  2^n\Vol(\mathcal{H}_{3,n})=-\frac{8}{(n+1)(2n+1)(2n+3)}+\frac{15\,\psi_n}{2(n+1)(2n+3)}\\
  +\frac{8\,(n)_3\,J_n}{(2n+1)(2n+2)(2n+3)}
  +\frac{24\,(n)_3\,K_n+12\,n(n-1)\,K'_n}{n(2n+1)(2n+2)(2n+3)}
\end{multline*}
and
\[
  2^n\mleft(\Vol(\mathcal{P}_{3,n}\cap \mathcal{D}_{3,n})-\Vol(\mathcal{H}_{3,n})\mright)
  =\frac{4\,(n)_3\,G_n}{(2n+1)(2n+2)(2n+3)}\,,
\]
where the terms with the factor $(n)_3$ are absent for $n\le2$ and
\begin{align*}
  J_n&:=\int_{S_3}\prod_{l\in M}\mleft(1-\sum_{i\ne l}v_i\mright)(1-|v|^2)^{n-3}\,dv,\\
  K_n&:=\int_{S_2}(1-v_1)(1-v_2)(1-v_1-v_2)(1-|v|^2)^{n-3}\,dv,\\
  K'_n&:=\int_{S_2}(1-v_1-v_2)(1-|v|^2)^{n-2}\,dv,\\
  G_n&:=\int_{S_3}w_1w_2w_3\,\phi(w)^{n-3}\,dw,
\end{align*}
with $\phi(w):=1-\frac34(1-W)^2-(1-W)W-|w|^2$ and $W:=w_1+w_2+w_3\,$.
\end{theorem}

\begin{proof}[Sketch of proof]
We integrate out $y$ using Theorem~\ref{thm:hull}, and we decompose according
to which columns attain the row minima $m_l\,$. With $r:=1-\sum_lm_l$ and
$c:=y-m$, this expresses $2^n\Vol(\mathcal{H}_{3,n})$ as an integral over $(r,c)\in[0,1]^4$
of a piecewise-polynomial integrand, restricted to $\sum_lc_l\ge2r$; dropping
the restriction gives $2^n\Vol(\mathcal{P}_{3,n}\cap \mathcal{D}_{3,n})$. The integrand has
$15$ polynomial pieces. The pieces in which all, or all but one, of the $c_l$
exceed $r$ reduce to Beta integrals and give the first two terms. After the
substitutions $u_l:=r-c_l$ and $u=rv$, the remaining pieces separate into a
Beta integral in $r$ and one of $J_n\,$, $K_n\,$, or $K'_n\,$. The pieces
outside $\sum_lc_l\ge2r$ give $G_n$ after the unimodular change of variables
$w_l:=1-\sum_{i\ne l}v_i\,$. The supplementary material contains the details,
together with independent checks: qhull for $\Vol(\mathcal{H}_{3,n})$ with $n\le3$, and
Theorem~\ref{thm:P} for $\Vol(\mathcal{P}_{3,n}\cap \mathcal{D}_{3,n})$ with $n\le4$.
\end{proof}

Each of $J_n\,$, $K_n\,$, $K'_n\,$, and $G_n$ is a finite sum of Dirichlet
moments, so the volumes are easy to compute exactly. For $n=1,\dots,5$, the
values of $2^n\Vol(\mathcal{H}_{3,n})$ are $7/30$, $13/105$, $61/756$, $3026/51975$, and
$1207/27027$, and those of $2^n\Vol(\mathcal{P}_{3,n}\cap \mathcal{D}_{3,n})$ are $7/30$,
$13/105$, $407/5040$, $1469/25200$, and $36833/823680$. Because $\mathcal{H}_{3,n}$ is a
lattice polytope, $(2n+3)!\,\Vol(\mathcal{H}_{3,n})$ is an integer; its first values are
$14$, $156$, $3660$, $145248$, and $8690400$.

For $m=3$, the second formula of Theorem~\ref{thm:m3vol} gives the exact rate at
which the hull and the strong relaxation approach each other.

\begin{proposition}\label{prop:m3sharp}
As $n\to\infty$,
\[
  \Vol(\mathcal{P}_{3,n}\cap \mathcal{D}_{3,n})-\Vol(\mathcal{H}_{3,n})
  \sim\frac{\pi}{12\sqrt3}\,\frac{1}{2^n\,n^2}\mleft(\frac23\mright)^{\!n},
\]
and hence
\[
  1-\frac{\Vol(\mathcal{H}_{3,n})}{\Vol(\mathcal{P}_{3,n}\cap \mathcal{D}_{3,n})}
  \sim\frac{2}{3\sqrt{3\pi}}\,\frac{1}{\sqrt n}\mleft(\frac23\mright)^{\!n}.
\]
\end{proposition}

\begin{proof}
By Theorem~\ref{thm:m3vol}, it suffices to find the asymptotics of $G_n\,$. On
$S_3\,$, the function $\phi$ attains its maximum $2/3$ at the single point
$w^*:=(\frac13,\frac13,\frac13)$, which lies on the face $W=1$. Indeed, writing
$u:=W-1\le0$ and splitting $w-w^*$ into its component along $(1,1,1)$ and a
component $\tau$ orthogonal to that direction, we have
\[
  \phi(w)=\frac23+\frac u3-\frac{u^2}{12}-|\tau|^2 .
\]
Hence
\[
  \phi(w)^{n-3}=\mleft(\frac23\mright)^{\!n-3}
  \exp\mleft((n-3)\ln\mleft(1+\frac u2-\frac32|\tau|^2+O(u^2)\mright)\mright),
\]
and Laplace's method applies, with the prefactor $w_1w_2w_3\to\frac1{27}$ and
with $dw=3^{-1/2}\,du\,d\tau$. Using $\int_{-\infty}^0e^{nu/2}\,du=2/n$ and
$\int_{\R^2}e^{-3n|\tau|^2/2}\,d\tau=2\pi/(3n)$, we obtain
\[
  G_n\sim\frac{1}{27}\mleft(\frac23\mright)^{\!n-3}\frac{1}{\sqrt3}\cdot\frac2n\cdot\frac{2\pi}{3n}
  =\frac{\pi}{6\sqrt3}\,\frac{1}{n^2}\mleft(\frac23\mright)^{\!n}.
\]
The first claim follows, because $4(n)_3/((2n+1)(2n+2)(2n+3))\to\frac12$. The
second follows from Theorem~\ref{thm:uniform}, because
$2^n\Vol(\mathcal{P}_{3,n}\cap \mathcal{D}_{3,n})\sim\psi(n,3)^3\sim(\pi^{3/2}/8)\,n^{-3/2}$.
\end{proof}

\begin{theorem}[Recurrences]\label{thm:m3rec}
The sequence $a_n:=2^n\Vol(\mathcal{H}_{3,n})$ satisfies, for all $n\ge1$,
\begin{multline*}
  -4(n+1)^2(n+2)(2n+1)(n^2-n-7)\,a_n\\
  +2(n+2)^2(2n+5)(9n^3-90n-77)\,a_{n+1}\\
  -(n+3)(2n+5)(2n+7)(13n^3+2n^2-164n-169)\,a_{n+2}\\
  +3(n+4)(2n+5)(2n+7)(2n+9)(n^2-3n-5)\,a_{n+3}=0,
\end{multline*}
and the sequence $b_n:=2^n\Vol(\mathcal{P}_{3,n}\cap \mathcal{D}_{3,n})$ satisfies, for all
$n\ge1$,
\begin{multline*}
  4(n+1)^2(n+2)(2n+1)(2n+3)\,b_n\\
  -20(n+2)^2(2n+3)(2n+5)^2\,b_{n+1}\\
  +5(n+3)(2n+5)^2(2n+7)(14n+43)\,b_{n+2}\\
  -2(n+4)(2n+5)(2n+7)(2n+9)(50n+181)\,b_{n+3}\\
  +24(n+5)(2n+5)(2n+7)(2n+9)(2n+11)\,b_{n+4}=0.
\end{multline*}
\end{theorem}

\begin{proof}[Computer-assisted proof]
We first found the recurrences by fitting them to exact values, and we then
proved them by integration by parts. For each integral family that occurs, we
computed a \emph{certificate}: a polynomial vector field $Q$ with
$(\operatorname{div}Q)F+(p+1)\,Q\cdot\nabla F$ equal to a prescribed
combination of integrands, where $F$ is the base of the power $F^p$. This is an
identity between polynomials, which we check exactly, and the divergence
theorem turns it into a first-order recurrence whose inhomogeneous term is the
flux of $QF^{p+1}$ through the boundary, that is, an integral of the next-lower
dimension. For $J$ and $G$, we chose the certificates so that the flux through
the face $\sum_iv_i=1$ is a constant multiple of $\int_{S_2}(2e_2)^{p+1}$, where
$e_2$ is the second elementary symmetric polynomial of $(v_1\,,v_2\,,1-v_1-v_2)$.
All two-dimensional families reduce, in this way, to one master integral each,
and all one-dimensional integrals are Beta integrals or satisfy first-order
recurrences with explicit values at the endpoints.

Applying each recurrence operator to the formulae of
Theorem~\ref{thm:m3vol} and reducing, we express the result as a combination of
the master integrals at a common exponent and of explicit terms, each a
rational function of $n$ times one of $1$,
$\sqrt\pi\,\Gamma(n)/\Gamma(n+\frac12)$,
$\sqrt\pi\,\Gamma(n)/(2^n\Gamma(n+\frac12))$, and $4^{-n}$. Every coefficient
cancels to $0$ identically. No coefficient of a certificate has a pole at a
nonnegative integer, so the argument is valid for $n\ge3$, and the cases
$n=1,2$ follow from the exact values. The supplementary file
\texttt{m3\_volumes\_supplement.ipynb} carries out all of these steps.
\end{proof}

\begin{remark}
For each recurrence, Petkov\v{s}ek's algorithm (as implemented in SymPy) finds, up
to scalar multiples, the single hypergeometric solution $\psi_n/(2n+3)$. Hence
neither volume is a finite linear combination of hypergeometric terms, and
Theorem~\ref{thm:m3vol} is, in this sense, as explicit as one can expect.
\end{remark}

We computed $\Vol(\mathcal{Q}_{3,n}\cap \mathcal{D}_{3,n})$ from Theorem~\ref{thm:Q}, and
$\Vol(\mathcal{P}_{3,n}\cap \mathcal{D}_{3,n})$ and $\Vol(\mathcal{H}_{3,n})$ by running the recurrences of
Theorem~\ref{thm:m3rec} forward in exact arithmetic from the initial values
listed there; for $n\le20$, the results agree with the formulae of
Theorem~\ref{thm:m3vol}. Figure~\ref{fig:m3} shows the three volumes for $m=3$
and $n$ up to $1000$; the dotted line in panel~(a) is the ratio predicted by
Theorem~\ref{thm:uniform}. The weak relaxation is larger than the integer hull
by a factor that grows like $n^{3/2}$, in agreement with
Theorem~\ref{thm:uniform}; at $n=1000$, the factor is about
$1.3\times10^{4}$. The strong relaxation, by contrast, is volumetrically almost
the integer hull. The two coincide for $n\le2$, their ratio is largest at
$n=5$, where the strong relaxation exceeds the hull by about $0.13$ percent,
and by $n=20$ the excess is below $10^{-5}$. So, for three facilities,
essentially all of the volume that separates the weak relaxation from the
integer hull is removed by the strong forcing constraints, and the odd-cycle
inequalities of Theorem~\ref{thm:hull} remove very little volume, although they
are needed for integrality once $n\ge3$.

\medskip
\noindent {\bf Declaration of generative AI and AI-assisted technologies in the manuscript preparation process.}
During the preparation of this work, the authors used Claude to assist in developing and and checking the work. 
The author takes full responsibility for the content of the published article.

\bibliographystyle{elsarticle-num}
\bibliography{texas_hot}

\end{document}